\documentclass[12pt,a4paper]{article}
\usepackage[margin=1in]{geometry}
\usepackage{libertine}
\usepackage[T1]{fontenc}
\usepackage[utf8]{inputenc}
\usepackage{amsfonts,amssymb,amsmath,amsthm}
 \usepackage[sort&compress]{natbib}
\usepackage{calc}
\usepackage{tikz}
\usetikzlibrary{calc}
\usepackage[bottom,para,symbol]{footmisc}
\usepackage{hyperref}
\usepackage{thmtools}
\usepackage{thm-restate}

\usepackage{pdfpages}
\usepackage{epigraph}

\usepackage{url}
\usepackage{hyperref}
\usepackage{enumitem}
\usepackage{xspace}
\usepackage{todonotes}
\usepackage{thmtools}
\usepackage{mathrsfs}  
\usepackage{textpos}
\usepackage[normalem]{ulem}
\usepackage[capitalize, nameinlink]{cleveref}

\DeclareMathOperator{\tw}{tw}

\newtheorem{theorem}{Theorem}[section]

\newtheorem{conjecture}[theorem]{Conjecture}
\theoremstyle{definition}

\def\N{\mathbb{N}}

\renewcommand{\leq}{\leqslant}
\renewcommand{\geq}{\geqslant}

\renewcommand{\mid}{~|~}

\usepackage{framed}
\usepackage{tabularx}

\newlength{\RoundedBoxWidth}
\newsavebox{\GrayRoundedBox}
\newenvironment{GrayBox}[1]%
   {\setlength{\RoundedBoxWidth}{.93\columnwidth}
    \def\boxheading{#1}
    \begin{lrbox}{\GrayRoundedBox}
       \begin{minipage}{\RoundedBoxWidth}}%
   {   \end{minipage}
    \end{lrbox}
    \begin{center}
    \begin{tikzpicture}%
       \node(Text)[draw=black!20,fill=white,rounded corners,inner sep=2ex,text width=\RoundedBoxWidth]
             {\usebox{\GrayRoundedBox}};
        \coordinate(x) at (current bounding box.north west);
        \node [draw=white,rectangle,inner sep=3pt,anchor=north west,fill=white]
        at ($(x)+(6pt,.75em)$) {\boxheading};
    \end{tikzpicture}
    \end{center}}

\newenvironment{defproblemx}[1]{\noindent\ignorespaces%
                                \FrameSep=6pt%
                                \parindent=0pt%
                \begin{GrayBox}{#1}%
                \begin{tabular*}{\columnwidth}{!{\extracolsep{\fill}}@{\hspace{.1em}} >{\itshape} p{1.5cm} p{0.86\columnwidth} @{}}%
            }{
                \end{tabular*}%
                \end{GrayBox}%
                \ignorespacesafterend
            }

\usepackage{mathtools}

\DeclareMathOperator{\treealpha}{tree\textnormal{-}\alpha}

\title{Tree-independence number VII. Excluding  a star.}
\author{
Maria Chudnovsky\thanks{Princeton University, Princeton, NJ, USA. 
Supported by NSF Grant DMS-2348219, NSF Grant CCF-2505100, AFOSR grant FA9550-22-1-0083 and a Guggenheim Fellowship.}
\and Jadwiga Czyżewska\thanks{University of Warsaw, Poland (\texttt{j.czyzewska@mimuw.edu.pl}).
Supported by Polish National Science Centre SONATA BIS-12 grant number 2022/46/E/ST6/00143.}
\and Marcin Pilipczuk\thanks{University of Warsaw, Poland (\texttt{m.pilipczuk@mimuw.edu.pl}).
Supported by Polish National Science Centre SONATA BIS-12 grant number 2022/46/E/ST6/00143.}
\and Paweł Rzążewski\thanks{Warsaw University of Technology \& University of Warsaw, Poland (\texttt{pawel.rzazewski@pw.edu.pl}). Supported by the National Science Centre grant 2024/54/E/ST6/00094.}}

\begin{document}
\date{}
\maketitle

\begin{abstract}
We prove that for every fixed integer $s$ and every planar graph $H$,
the class of $H$-induced-minor-free and $K_{1,s}$-induced-subgraph-free graphs has polylogarithmic tree-independence number.
This is a weakening of a conjecture of Dallard, Krnc, Kwon, Milani\v{c}, Munaro, \v{S}torgel, and Wiederrecht.

\end{abstract}

\section{Introduction}

Tree decompositions and treewidth are among the most influential concepts in structural graph theory.
Intuitively, a tree decomposition is a hierarchical decomposition of a graph $G$ into sets called \emph{bags}.
If these sets are all small (i.e., $G$ has small treewidth), then $G$ is ``tree-like'' and thus ``simple;'' see \cref{sec:defs} for a formal definition.

Since its birth, the notion of treewidth was closely related to graph minors.
(A graph $H$ is a~minor of a graph $G$ if it can be obtained from $G$ by deleting vertices and edges, and contracting edges.)
This close relation is witnessed by the following landmark result of Robertson and Seymour~\cite{GMV}, usually referred to as \emph{Grid Minor Theorem}.

\begin{theorem}[Robertson, Seymour~\cite{GMV}]\label{thm:gridminor}
    For every planar graph $H$ there exists an integer $c_H$ such that every graph that does not contain a minor isomorphic to $H$ has treewidth at most $c_H$.
\end{theorem}

On the other hand, any class that does not exclude a planar graph contains all planar graphs which have unbounded treewidth.
Thus, \cref{thm:gridminor} provides a full characterization of minor-closed classes that have bounded treewidth.

While graphs that exclude some fixed graph as a minor are necessarily sparse,
it turns out that tree decompositions can also find application in the study of well-behaved classes of dense graphs.
A class of graphs is {\em hereditary} if it is closed under vertex deletion.
Let $G$ and $H$ be graphs.
We say that $H$ is an \emph{induced subgraph} of $G$ if it can be obtained from $G$ by removing vertices.
If $H$ is not an induced subgraph of $G$, then $G$ is $H$-free.
We say that $H$ is an {\em induced minor} of $G$ if
$H$ can be obtained from an induced subgraph of $G$ by contracting edges (and repeatedly deleting parallel edges obtained in the process).

In recent years a lot of attention attention was devoted to the study of treewidth of hereditary graph classes.
Again, the question is the same: Which substructures should one exclude to obtain a class of bounded treewidth?
Despite significant progress on this question~\cite{AbrishamiAlecuHajebiSpirkl2025BasicObstructionsFiniteH,
  HajebiSpirkl2025NonAdjacentNeighborsHole,AbrishamiAlecuHajebiSpirkl2024ExcludingForest,  AbrishamiAlecuHajebiSpirkl2024BasicObstructions,AbrishamiHajebiSpirkl2022ThreePathConfigurations, BHKM}, we are still quite far from a full resolution.
However, the answer is known if we additionally assume that the maximum vertex degree is bounded.
Indeed, Korhonen~\cite{korhonen2023grid} proved the following analogue of \cref{thm:gridminor}, which was earlier conjectured by Aboulker et al.~\cite{aboulker}.

\begin{theorem}[Korhonen~\cite{korhonen2023grid}]\label{thm:Deltagridminor}
    For every integer $\Delta$ and a planar graph $H$ there exists an integer $c_{\Delta,H}$ such that every graph of maximum degree at most $\Delta$ that does not contain an induced minor isomorphic to $H$ has treewidth at most $c_{\Delta,H}$.
\end{theorem}

Another way of dealing with dense graphs is to redefine how we measure the quality of a tree decomposition.
Instead of saying that a graph is ``simple'' if it has a tree decomposition where each bag is small,
we can instead ask for tree decompositions where every bag induces a subgraph of ``simple structure.''
For example chordal graphs are precisely the ones that admit a tree decomposition where every bag is a clique.
This leads to the notion of \emph{tree-independence number}, another graph parameter associated with tree decompositions, introduced independently by Yolov~\cite{DBLP:conf/soda/Yolov18} and by Dallard, Milani\v{c}, and \v{S}torgel~\cite{dms2}.
Intuitively, the tree-independence number of $G$, is the minimum $k$ such that $G$ has a tree decomposition where no bag contains $k+1$ pairwise non-adjacent vertices.
For example, aforementioned chordal graphs are precisely graphs with tree-independence number 1.

Much of the research on tree-independence number revolves around trying to characterize graph families where this parameter is bounded, or at least grows slowly as a function of the size of the graph.
In this spirit, Dallard, Krnc, Kwon, Milani\v{c}, Munaro, \v{S}torgel, and Wiederrecht~\cite{DBLP:journals/corr/abs-2402-11222} suggested the following ``dense'' analogue of \cref{thm:Deltagridminor}.
(For integers $s,t$, by $K_{s,t}$  we denote the complete bipartite graph with sides of a bipartition of size $s$ and $t$.)
 

\begin{conjecture}[\cite{DBLP:journals/corr/abs-2402-11222}] \label{The Conjecture}
For every integer $s$ and every planar graph $H$ there exists an integer $c_{s,H}$
such that every graph  which is $H$-induced minor-free and $K_{1,s}$-free
has tree independence number at most $c_{s,H}$.
\end{conjecture}

The conjecture has been confirmed only  for very restricted cases~\cite{DBLP:journals/corr/abs-2402-11222,HMV25,DBLP:journals/corr/abs-2506-08829, ti3}. In this short note we prove a polylogarithmic version of \cref{The Conjecture}.

\begin{restatable}{theorem}{thmmain}\label{thm:main}
For every integer $s$ and every planar graph $H$ there exists an constant $c_{s,H}$
such that every $n$-vertex graph which is $H$-induced minor-free and $K_{1,s}$-free
has tree-independence number at most $\log^{c_{s,H}} n$.
\end{restatable}

\section{Notation and tools}\label{sec:defs}

\paragraph{Graphs.}
An \emph{independent set} is a subset of vertices of $V(G)$ which are pairwise non-adjacent.
The \emph{independence number} of set $A \subseteq V(G)$, denoted by $\alpha(A)$, is the size of the largest independent set in $G[A]$.

A \emph{clique} in $G$ is a set of vertices of $G$ that are pairwise adjacent. The \emph{clique number} of a~graph $G$, denoted by $\omega(G)$,
is the number of vertices in a largest clique of $G$.

We will use the following bound for the off-diagonal Ramsey number.
\begin{theorem}[Ramsey~\cite{Ramsey}, see also Erd{\H o}s-Szekeres~\cite{ErdosSzekeres:1935:ACombinatorialProblemInGeometry}] \label{Ramsey}
\label{thm:ramsey}
For all $s,t \in \N$, every graph on at least $t^s$ vertices has either a clique of cardinality $t$ or an independent set  of cardinality $s+1$.
\end{theorem}

An $r\times r$ hexagonal grid is denoted as $W_{r\times r}$.
The following result is folklore, see e.g. \cite[Theorem 12]{campbell2024treewidthhadwigernumberinduced}.
\begin{theorem} \label{HtoW}
For every planar graph $H$ there exist $r \in \N$ such that $H$ is an induced minor of
$W_{r \times r}$.
\end{theorem}

\paragraph{Tree decompositions.}

A tree decomposition $\mathcal{T}$ of a graph $G$ is a pair $(T,\beta)$ where $T$ is a tree and $\beta$ is a function assigning each node of $T$ a non-empty subset $V(G)$ such that the following conditions are satisfied:
\begin{enumerate}
\itemsep -.2em
    \item For each vertex $v$ of $V(G)$ a subset of nodes $\{x\in V(T)\mid v\in\beta(x)\}$ induces a  non-empty subtree;
    \item For each edge $uv$ of $E(G)$ there exists a node $x\in V(T)$ such that $u,v\in \beta(x)$.
\end{enumerate}

The \emph{width} of a tree decomposition $\mathcal{T}=(T, \beta)$ is equal to $\max_{x\in V(T)}\ |\beta(x)|-1$.
The \emph{treewidth} of a graph $G$ is a minimal width over all tree decompositions of $G$ and is denoted as $\tw(G)$.
The \emph{independence number} of a tree decomposition $\mathcal{T}=(T,\beta)$ is equal to $\max_{x\in V(T)} \alpha(\beta(x))$.
The \emph{tree-independence number} of a graph $G$ is a minimal independence number over all tree decomposition of $G$ and is denoted as $\treealpha(G)$.

\subsection{Building blocks}

The proof of  relies on three results from the literature.
We start the definitions necessary to state these results.
The first one is a theorem describing properties of graphs that contains a
large complete bipartite graph as an induced minor.

A \emph{constellation}, defined in \cite{Chudnovsky_2026},  is a graph $\mathfrak{c}$ in which there is an independent set $I_\mathfrak{c}$ such that each connected component in $\mathfrak{c}-I_\mathfrak{c}$ is a path and each vertex $v\in I_\mathfrak{c}$ has at least one neighbor in each connected component in $\mathfrak{c}-I_\mathfrak{c}$.
An \emph{$(s,\ell)$-constellation} is a constellation $\mathfrak{c}$ where $|I_\mathfrak{c}|=s$ and there are $\ell$ connected components in $\mathfrak{c}-I_\mathfrak{c}$.
 We can now state the first theorem that we need.

\begin{theorem}[{Chudnovsky, Hajebi, Spirkl~\cite[Theorem 1.3]{chudnovsky2025inducedsubgraphstreedecompositions}}]
\label{thm:constellation}
    For all $\ell,r,q\in \mathbb{N}$, there is a~constant $t \in \mathbb{N}$ such that if $G$ is a graph with an induced minor isomorphic to $K_{t,t}$, then one of the following holds.
 \begin{enumerate}
        \item There is an induced minor of $G$ isomorphic to $W_{r\times r}$.
        \item There is an $(q,\ell)$-constellation in $G$.
    \end{enumerate}
\end{theorem}

For $\lambda \in \N$, we say that a graph  $G$ is \emph{$\lambda$-separable} if for all pairs of vertices $u,\ v$ of $V(G)$,
which are distinct and non-adjacent,
there is no set of $\lambda$ pairwise internally disjoint paths in $G$ from $u$ to $v$.
The next result that we use is the following:

\begin{theorem}[{Hajebi~\cite[Theorem 3.2 for $\kappa=2$]{hajebi2025polynomialboundspathwidth}}]\label{thm:hajebi}
    For every planar graph $H$ and every $t$ there exists $d\in \N$ such that for all $\lambda \in \N$,
    if $G$ is a $\lambda$-separable graph with no induced minor isomorphic to $H$ or $K_{t,t}$,
    then $\tw(G)\leq (2 (\omega(G)+1))^d$.
\end{theorem}

We also need:
\begin{theorem}[Chudnovsky,  Lokshtanov, Satheeshkumar~\cite{chudnovsky2025treewidthcliqueboundednesspolylogarithmictreeindependence}]\label{thm:treeaplha}
Let $\mathcal{C}$ be a hereditary class. Then the following are equivalent:
\begin{enumerate}
    \item There exists a positive constant $c_1$ such that for every graph $G\in \mathcal{C}$ on $n \geq 3$ vertices we~have $\treealpha(G)\leq (\log n)^{c_1}$.
    \item There exists a positive constant $c_2$ such that for every graph $G\in \mathcal{C}$ on $n \geq 3$ vertices we~have $\treealpha(G)\leq (\omega(G)\log n)^{c_2}$.
    \item There exists a positive constant $c_3$ such that for every graph $G\in \mathcal{C}$ on $n \geq 3$ vertices we~have $\tw(G)\leq (\omega(G)\log n)^{c_3}$.
\end{enumerate}
\end{theorem}

\section{Proof of \cref{thm:main}}

\thmmain*
\begin{proof}
    Given $H$ and $s$, let us consider any $n$-vertex graph $G$ which is $H$-induced-minor-free and $K_{1,s}$-free. Let us denote the clique number of $G$ as $\omega$.
    Since $H$ is planar, by \cref{HtoW}  there exists $r$ such that $H$ is an induced  minor of $W_{r \times r}$.
    Thus, $G$ excludes $W_{r \times r}$ as an induced minor.

    Since $G$ is $K_{1,s}$-free, it follows that no induced subgraph of $G$ is a $(1,s)$-constellation.
    Applying \cref{thm:constellation} with $q=1$ and $\ell=s$, we deduce that there is  $t \in \N$ (that depends on $H$ and $s$ only),
    such that $G$ is  $K_{t,t}$-induced-minor-free.
     
    Denote by $\Delta$ the maximum degree of a vertex in $G$. 
    For every vertex $v\in V(G)$ there is no independent set of size $s$ or a clique of size $\omega$ inside $N(v)$. 
    Thus by \cref{Ramsey} we get that $\Delta(G)< \omega^s$.
    Since for every pair of vertices  $u,v$ in  $G$ there exist at most  $\Delta$ pairwise vertex disjoint paths from $u$ to $v$,
    it follows that $G$ is $(\Delta(G)+1)$-separable.
    Consequently, $G$ is $\omega^s$-separable
     
    Since $G$ is $\omega^s$-separable and $K_{t,t}$-induced-minor-free,
    \cref{thm:hajebi} implies that  there exists $d$ that depends only on $t$ and $H$,
    and therefore only on $s$ and $H$, such that   
     \[\tw(G)\leq \left(2 \omega^s(\omega+1)\right)^d.\]
    Finally, by \cref{thm:treeaplha} we  get that $\treealpha(G)\leq \log^c n$ where $c$ is a constant that depends only on $s$ and $H$.  
    This completes the proof.
\end{proof}

\section{Conclusion}

In this note we proved that \cref{The Conjecture} is ``morally true,'' i.e., it holds up to factors polylogaritmic in the number of vertices.
The full resolution of the conjecture is wide open.

Let us remark that Dallard et al.~\cite{dms3} made another conjecture about tree-independence number -- they suggested that every hereditary class where treewidth is bounded in terms of the clique number, has bounded tree-independence number.
This conjecture was recently refuted by Chudnovsky and Trotignon~\cite{Chudnovsky_2025}.
However, shortly after that Chudnovsky,  Lokshtanov, and Satheeshkumar~\cite{chudnovsky2025treewidthcliqueboundednesspolylogarithmictreeindependence} proved that the conjecture is ``morally true'' (again, up to polylogarithmic factors), see \cref{thm:treeaplha}.

Let us conclude the paper with recalling yet another conjecture by Dallard, Krnc, Kwon, Milani\v{c}, Munaro, \v{S}torgel, and Wiederrecht~\cite{DBLP:journals/corr/abs-2402-11222}, closely related to \cref{The Conjecture}.

\begin{conjecture}[\cite{DBLP:journals/corr/abs-2402-11222}]
Let $\mathcal{S}$ denote the family of forests where every component has at most three leaves.
For every $S \in \mathcal{S}$ and every integer $t$ there exists $c_{S,t}$ such that every graph which is $S$-induced-minor-free and $K_{t,t}$-free has tree-independence number $C_{S,t}$.
\end{conjecture}

Interestingly, as shown by Chudnovsky et al.~\cite{DBLP:journals/corr/abs-2501-14658}, this conjecture is also ``morally true,'' i.e., all such graphs have tree-independence number polylogarithmic in the number of vertices.

\bibliographystyle{plain}
\bibliography{bibliography}

@article{Chudnovsky_2025,
   title={On treewidth and maximum cliques},
   volume={2},
   ISSN={3050-743X},
   url={http://dx.doi.org/10.5802/igt.11},
   DOI={10.5802/igt.11},
   journal={Innovations in Graph Theory},
   publisher={Cellule MathDoc/Centre Mersenne},
   author={Chudnovsky, Maria and Trotignon, Nicolas},
   year={2025},
   month=sep, pages={223–243} }

@article{korhonen2023grid,
  title={Grid induced minor theorem for graphs of small degree},
  author={Korhonen, Tuukka},
  journal={Journal of Combinatorial Theory, Series B},
  volume={160},
  pages={206--214},
  year={2023},
  publisher={Elsevier}
}

@article{aboulker,
  title={On the tree-width of even-hole-free graphs},
  author={Aboulker, Pierre and Adler, Isolde and Kim, Eun Jung and Sintiari, Ni Luh Dewi and Trotignon, Nicolas},
  journal={European Journal of Combinatorics},
  volume={98},
  pages={103394},
  year={2021},
  publisher={Elsevier}
}

@article{dms3,
    author       = {Cl{\'{e}}ment Dallard and
                  Martin Milanic and
                  Kenny Storgel},
title = {Treewidth versus clique number. {III. Tree-independence number of graphs
 with a forbidden structure}},
journal = {Journal of Combinatorial Theory, Series B},
volume = {167},
pages = {338-391},
year = {2024},
}

@article{GMV,
  title={Graph minors. {{V}}. {{E}}xcluding a planar graph},
  author={Robertson, Neil and Seymour, Paul},
  journal={Journal of Combinatorial Theory, Series B},
  volume={41},
  number={1},
  pages={92--114},
  year={1986},
  publisher={Elsevier}
}

@misc{ti3,
      title={{Tree independence number III. Thetas, prisms and stars}}, 
      author={Maria Chudnovsky and Sepehr Hajebi and Nicolas Trotignon},
      year={2025},
      eprint={2406.13053},
      archivePrefix={arXiv},
      primaryClass={math.CO},
      url={https://arxiv.org/abs/2406.13053}, 
}

@article{chudnovsky2025treewidthcliqueboundednesspolylogarithmictreeindependence,
  author     = {Maria Chudnovsky and
                Ajaykrishnan E. S. and
                Daniel Lokshtanov},
  title      = {(Treewidth, Clique)-Boundedness and Poly-logarithmic Tree-Independence},
  journal    = {CoRR},
  volume     = {abs/2510.15074},
  year       = {2025},
  url        = {https://arxiv.org/abs/2510.15074},
  eprint     = {2510.15074},
  eprinttype = {arXiv}
}

@article{chudnovsky2025inducedsubgraphstreedecompositions,
  author     = {Maria Chudnovsky and
                Julien Codsi and
                Sepehr Hajebi and
                Sophie Spirkl},
  title      = {{Induced Subgraphs and Tree Decompositions {XIX}. Thetas and Forests}},
  journal    = {CoRR},
  volume     = {abs/2506.05602},
  year       = {2025},
  url        = {https://arxiv.org/abs/2506.05602},
  eprint     = {2506.05602},
  eprinttype = {arXiv}
}

@article{hajebi2025polynomialboundspathwidth,
  author     = {Sepehr Hajebi},
  title      = {Polynomial Bounds for Pathwidth},
  journal    = {CoRR},
  volume     = {abs/2510.19120},
  year       = {2025},
  url        = {https://arxiv.org/abs/2510.19120},
  eprint     = {2510.19120},
  eprinttype = {arXiv}
}

@article{Chudnovsky_2026,
  author  = {Maria Chudnovsky and
             Sepehr Hajebi and
             Sophie Spirkl},
  title   = {{Induced Subgraphs and Tree Decompositions {XVI}. Complete Bipartite Induced Minors}},
  journal = {Journal of Combinatorial Theory, Series B},
  volume  = {176},
  pages   = {287--318},
  year    = {2026},
  doi     = {10.1016/j.jctb.2025.09.005},
  url     = {https://doi.org/10.1016/j.jctb.2025.09.005}
}

@article{HMV25,
  author     = {Claire Hilaire and
                Martin Milani\v{c} and
                {\DJ}or{\dj}e Vasi\v{c}},
  title      = {{Treewidth versus Clique Number. {V}. Further Connections with Tree-Independence Number}},
  journal    = {CoRR},
  volume     = {abs/2505.12866},
  year       = {2025},
  url        = {https://arxiv.org/abs/2505.12866},
  eprint     = {2505.12866},
  eprinttype = {arXiv}
}

@article{Ramsey,
  author  = {Frank P. Ramsey},
  title   = {On a Problem of Formal Logic},
  journal = {Proceedings of the London Mathematical Society},
  series  = {2},
  volume  = {30},
  number  = {4},
  pages   = {264--286},
  year    = {1929},
  doi     = {10.1112/plms/s2-30.1.264},
  url     = {https://doi.org/10.1112/plms/s2-30.1.264}
}

@article{campbell2024treewidthhadwigernumberinduced,
  author     = {Rutger Campbell and
                James Davies and
                Marc Distel and
                Bryce Frederickson and
                J. Pascal Gollin and
                Kevin Hendrey and
                Robert Hickingbotham and
                Sebastian Wiederrecht and
                David R. Wood and
                Liana Yepremyan},
  title      = {Treewidth, {Hadwiger} Number, and Induced Minors},
  journal    = {CoRR},
  volume     = {abs/2410.19295},
  year       = {2024},
  url        = {https://arxiv.org/abs/2410.19295},
  eprint     = {2410.19295},
  eprinttype = {arXiv}
}

@article{dms2,
  author  = {Cl{\'e}ment Dallard and
             Martin Milani\v{c} and
             Kenny \v{S}torgel},
  title   = {{Treewidth versus Clique Number. {II}. Tree-Independence Number}},
  journal = {Journal of Combinatorial Theory, Series B},
  volume  = {164},
  pages   = {404--442},
  year    = {2024},
  doi     = {10.1016/j.jctb.2023.10.006},
  url     = {https://doi.org/10.1016/j.jctb.2023.10.006}
}

@article{DBLP:journals/corr/abs-2402-11222,
  author     = {Cl{\'e}ment Dallard and
                Matja\v{z} Krnc and
                O{-}joung Kwon and
                Martin Milani\v{c} and
                Andrea Munaro and
                Kenny \v{S}torgel and
                Sebastian Wiederrecht},
  title      = {{Treewidth versus Clique Number. {IV}. Tree-Independence Number of Graphs Excluding an Induced Star}},
  journal    = {CoRR},
  volume     = {abs/2402.11222},
  year       = {2024},
  url        = {https://doi.org/10.48550/arXiv.2402.11222},
  doi        = {10.48550/ARXIV.2402.11222},
  eprint     = {2402.11222},
  eprinttype = {arXiv}
}

@article{DBLP:journals/corr/abs-2506-08829,
  author     = {Mujin Choi and
                Claire Hilaire and
                Martin Milani\v{c} and
                Sebastian Wiederrecht},
  title      = {Excluding an Induced Wheel Minor in Graphs without Large Induced Stars},
  journal    = {CoRR},
  volume     = {abs/2506.08829},
  year       = {2025},
  url        = {https://doi.org/10.48550/arXiv.2506.08829},
  doi        = {10.48550/ARXIV.2506.08829},
  eprint     = {2506.08829},
  eprinttype = {arXiv}
}

@inproceedings{DBLP:conf/soda/Yolov18,
  author    = {Nikola Yolov},
  title     = {Minor-Matching Hypertree Width},
  booktitle = {Proceedings of the Twenty-Ninth Annual {ACM-SIAM} Symposium on Discrete Algorithms ({SODA} 2018)},
  pages     = {219--233},
  publisher = {{SIAM}},
  year      = {2018},
  doi       = {10.1137/1.9781611975031.16},
  url       = {https://doi.org/10.1137/1.9781611975031.16}
}

@article{ErdosSzekeres:1935:ACombinatorialProblemInGeometry,
  author  = {Paul Erd{\H o}s and
             George Szekeres},
  title   = {A Combinatorial Problem in Geometry},
  journal = {Compositio Mathematica},
  volume  = {2},
  pages   = {463--470},
  year    = {1935},
  url     = {https://www.numdam.org/item/CM_1935__2__463_0/}
}

@article{BHKM,
      title={Treewidth is Polynomial in Maximum Degree on Weakly Sparse Graphs Excluding a Planar Induced Minor}, 
      author={\'Edourad Bonnet and Jędrzej Hodor and Tuukka Korhonen and Tomáš Masařík},
      year={2024},
journal={{\rm preprint available at \url{https://arxiv.org/abs/2312.07962}}}, 
}

@article{DBLP:journals/corr/abs-2501-14658,
  author       = {Maria Chudnovsky and
                  Julien Codsi and
                  Daniel Lokshtanov and
                  Martin Milani\v{c} and
                  Varun Sivashankar},
  title        = {Tree independence number V. Walls and claws},
  journal      = {CoRR},
  volume       = {abs/2501.14658},
  year         = {2025},
  url          = {https://doi.org/10.48550/arXiv.2501.14658},
  doi          = {10.48550/ARXIV.2501.14658},
  eprinttype    = {arXiv},
  eprint       = {2501.14658},
  bibsource    = {dblp computer science bibliography, https://dblp.org}
}

@article{AbrishamiAlecuHajebiSpirkl2025BasicObstructionsFiniteH,
  author  = {Abrishami, Tara and Alecu, Bogdan and Hajebi, Sepehr and Spirkl, Sophie},
  title   = {{Induced Subgraphs and Tree-Decompositions XIII. Basic Obstructions in $H$-Free Graphs for Finite $H$}},
  journal = {Advances in Combinatorics},
  year    = {2025}
}

@article{HajebiSpirkl2025NonAdjacentNeighborsHole,
  author  = {Chudnovsky, Maria and Hajebi, Sepehr and Spirkl, Sophie},
  title   = {{Induced Subgraphs and Tree-Decompositions XIV. Non-Adjacent Neighbors in a Hole}},
  journal = {European Journal of Combinatorics},
  volume  = {124},
  pages   = {104074},
  year    = {2025}
}

@article{AbrishamiAlecuHajebiSpirkl2024ExcludingForest,
  author  = {Abrishami, Tara and Alecu, Bogdan and Hajebi, Sepehr and Spirkl, Sophie},
  title   = {{Induced Subgraphs and Tree-Decompositions VIII. Excluding a Forest in (Prism, Theta)-Free Graphs}},
  journal = {Combinatorica},
  volume  = {44},
  pages   = {921--948},
  year    = {2024}
}

@article{AbrishamiAlecuHajebiSpirkl2024BasicObstructions,
  author  = {Abrishami, Tara and Alecu, Bogdan and Hajebi, Sepehr and Spirkl, Sophie},
  title   = {{Induced Subgraphs and Tree-Decompositions VII. Basic Obstructions in $H$-Free Graphs}},
  journal = {Journal of Combinatorial Theory, Series B},
  volume  = {164},
  pages   = {443--472},
  year    = {2024}
}

@article{AbrishamiHajebiSpirkl2022ThreePathConfigurations,
  author  = {Abrishami, Tara and Hajebi, Sepehr and Spirkl, Sophie},
  title   = {{Induced Subgraphs and Tree-Decompositions III. Three-Path-Configurations and Logarithmic Tree-Width}},
  journal = {Advances in Combinatorics},
  year    = {2022}
}

\end{document}